\documentclass{amsart}
\usepackage[total={6in, 8in}]{geometry}
\usepackage{graphicx} 
\usepackage{amsfonts}
\usepackage{amssymb, dirtytalk}
\usepackage{color}
\usepackage{amsmath}
\usepackage{amsthm}
\usepackage{multicol}
\usepackage{mathtools}
\usepackage{tikz,tikz-cd, tcolorbox}
\usepackage[colorlinks=true,linkcolor=blue,citecolor=blue]{hyperref}
\usepackage{setspace, mathrsfs}
\usepackage{bussproofs}
\theoremstyle{plain}
\newtheorem{theorem}{Theorem}[section]
\newtheorem{lemma}[theorem]{Lemma}
\newtheorem{proposition}[theorem]{Proposition}
\newtheorem{corollary}[theorem]{Corollary}
\usepackage{amsaddr}

\usepackage{graphicx}
\usepackage{multicol,multirow}
\usepackage{rotating}
\usepackage{appendix}
\usepackage{longtable}
\usepackage[sort&compress,numbers]{natbib}

\usepackage{xcolor,multicol}
\newcommand{\rul}[1]{\ensuremath{\mathrm{(#1)}}}
\newcommand{\Hom}{\ensuremath{\mathrm{Hom}}}
\newcommand{\tuple}[1]{\ensuremath{\langle{#1}\rangle}}
\newcommand{\alg}[1]{{\ensuremath{\boldsymbol {\mathit{#1}}}}}
\long\def\fnote#1{\bgroup\color{olive} (Filip: #1) \egroup}

\theoremstyle{definition}
\newtheorem{definition}[theorem]{Definition}
\newtheorem{remark}[theorem]{Remark}
\newtheorem{example}[theorem]{Example}

\begin{document}
\title[The pairwise Stone space of an S4 De Morgan algebra]{The pairwise Stone space of an S4 De Morgan algebra}
\author{Joseph McDonald and Filip Jankovec}
\address{Institute of Computer Science, Czech Academy of Sciences}

\email{mcdonald@cs.cas.cz, jankovec@cs.cas.cz}
\date{July 2026}
\maketitle

\begin{abstract}
    The purpose of this study is to investigate the bitopological duality theory of De Morgan algebras equipped with a closure operator, known as \emph{S4 De Morgan algebras}. We first introduce certain expansions of pairwise Stone spaces, which we call \emph{pairwise S4 De Morgan Stone spaces} (henceforth, \emph{PS4D-spaces}). These consist of a pairwise Stone space $X$ equipped with  a twist continuous involution $g\colon X\to X$, as well as a binary relation $R\subseteq X\times X$ that is reflexive and transitive. We first demonstrate that the bitopological spectrum $\mathcal{S}_0(A)$ of prime filters of an S4 De Morgan algebra $A$ gives rise to a PS4D-space. A topological representation is then obtained by exhibiting an isomorphism from $A$ to the S4 De Morgan algebra $\mathcal{A}_0(\mathcal{S}_0(A))$ of $(\tau_1,\delta_2)$-biclopen subsets of $\mathcal{S}_0(A)$ whose operation of De Morgan involution is defined through $g$ and whose closure operator is defined through $R$. We then provide an algebraic realization theorem by showing that every PS4D-space $X$ is bihomeomorphic and relationally isomorphic to the bitopological spectrum $\mathcal{S}_0(\mathcal{A}_0(X))$ of prime filters of $\mathcal{A}_0(X)$. With the introduction of suitable bicontinuous frame morphisms, we show that the category $\mathbf{S4D}$ of S4 De Morgan algebras is dually equivalent to the category $\mathbf{PStone_{S4D}}$ of PS4D-spaces. As an application, we provide bitopological characterizations of filters and ideals in general De Morgan algebras under our established duality as well as bitopological soundness and completeness results for an S4-type modal extension of the calculus FDE of first-degree entailment.
\end{abstract}

\section{Introduction}

Stone duality for bounded distributive lattices \cite{stone2} asserts that the category $\mathbf{DLatt_{01}}$ of bounded distributive lattices and bounded lattice homomorphisms is dually equivalent to the category $\mathbf{Spec}$ of spectral spaces and continuous functions whose inverse image of a compact set is compact. Priestley \cite{priestly} provided an alternative duality for $\mathbf{DLatt_{01}}$ relative to that of Stone by introducing certain partially ordered Stone spaces, known as \emph{Priestley spaces}, and showed that $\mathbf{DLatt_{01}}$ is dually equivalent to the category $\mathbf{Priest}$ of Priestley spaces and continuous monotone functions. Since $\mathbf{Spec}$ and $\mathbf{Priest}$ are both dually equivalent to $\mathbf{DLatt_{01}}$, it follows that $\mathbf{Spec}$ is equivalent to $\mathbf{Priest}$. It was later shown by Cornish \cite{cornish} (see also Fleisher \cite{fleisher}) that $\mathbf{Spec}$ is in fact isomorphic to $\mathbf{Priest}$. More recently, N. Bezhanishvili, G. Bezhanishvili, Gabelaia, and Kurz \cite{bezhanishvili} extended this isomorphism to the category $\mathbf{PStone}$ of pairwise Stone spaces and bicontinuous functions, thereby providing an alternative duality for $\mathbf{DLatt_{01}}$ relative to that of Stone and Priestley.

This three-way isomorphism between $\mathbf{Spec}$, $\mathbf{Priest}$, and $\mathbf{PStone}$ suggests that existing dualities for bounded distributive lattice-based algebras, which are often based on Priestley spaces or spectral spaces, admit of alternative bitopological dualities using pairwise Stone spaces. Under this motivation, N. Bezhanishvili, G. Bezhanishvili, Gabelaia, and Kurz \cite{bezhanishvili} translated Esakia duality for Heyting algebras, which \say{piggy-backs} off of Priestley duality, under the isomorphism between $\mathbf{Priest}$ and $\mathbf{PStone}$ by showing that the category $\mathbf{Heyt}$ of Heyting algebras and Heyting algebra homomorphisms is dually equivalent to the category $\mathbf{PEsakia}$ of pairwise Esakia spaces and bicontinuous Esakia morphisms. Likewise, Das and Ray \cite{das} provided bitopological duality results for the Heyting algebra expansions corresponding to Fitting's logic.

Recently, McDonald \cite{mcdonald} provided spectral duality results for S4 De Morgan algebras and De Morgan groupoids by appropriately adapting and extending the Priestley dualities for general De Morgan algebras developed by Cornish and Fowler \cite{cornish1} and Bimb\'o \cite{bimbo} as well as the Priestley duality for relevance algebras developed by Urquhart \cite{urq} under the isomorphism between $\mathbf{Priest}$ and $\mathbf{Spec}$. A \emph{De Morgan algebra} is a bounded distributive lattice $A$ equipped with an involution $-\colon A\to A$ satisfying De Morgan's identities. An \emph{S4 De Morgan algebra} consists of a De Morgan algebra $A$ equipped with a closure operator $\nabla\colon A\to A$. De Morgan algebras play an important role in non-classical logic as they form the algebraic counterpart to the relevance logic FDE developed by Belnap and Dunn \cite{belnap,Dunn:SemanticsBD}. The purpose of this study is to investigate the bitopological duality theory of S4 De Morgan algebras by translating the spectral duality results obtained in \cite{mcdonald} under the isomorphism between $\mathbf{Spec}$ and $\mathbf{PStone}$. We first introduce certain expansions of pairwise Stone spaces, which we call \emph{pairwise S4 De Morgan Stone spaces} (henceforth, \emph{PS4D-spaces}). These consist of a pairwise Stone space $X$ equipped with a twist continuous involution $g\colon X\to X$, as well as a binary relation $R\subseteq X\times X$ that is reflexive and transitive. We first demonstrate that the bitopological spectrum $\mathcal{S}_0(A)$ of prime filters of an S4 De Morgan algebra $A$ gives rise to a PS4D-space. A topological representation is then obtained by exhibiting an isomorphism from $A$ to the S4 De Morgan algebra $\mathcal{A}_0(\mathcal{S}_0(A))$ of $(\tau_1,\delta_2)$-biclopen subsets of $\mathcal{S}_0(A)$ whose operation of De Morgan involution is defined through $g$ and whose closure operator is defined through $R$. We then provide an algebraic realization theorem by showing that every PS4D-space $X$ is bihomeomorphic and relationally isomorphic to the bitopological spectrum $\mathcal{S}_0(\mathcal{A}_0(X))$ of prime filters of $\mathcal{A}_0(X)$. With the introduction of suitable bicontinuous frame morphisms, we show that the category $\mathbf{S4D}$ of S4 De Morgan algebras is dually equivalent to the category $\mathbf{PStone_{S4D}}$ of PS4D-spaces. As an application, we provide bitopological characterizations of filters and ideals in general De Morgan algebras under our established duality as well as bitopological soundness and completeness results for an S4-type modal extension of the calculus FDE of first-degree entailment.      

The contents of this paper are organized in the following manner: In Section \ref{sec:preliminaries}, we describe some basic facts about De Morgan algebras as well as S4 De Morgan algebras. In Section \ref{sec:representation}, we introduce pairwise S4 De Morgan Stone spaces and provide the promised bitopological representation of S4 De Morgan algebras. In Section \ref{sec:duality}, we provide an algebraic realization for pairwise S4 De Morgan Stone spaces and demonstrate that the category $\mathbf{S4D}$ of S4 De Morgan algebras is dually equivalent to the category $\mathbf{PStone_{S4D}}$ of pairwise S4 De Morgan Stone spaces. In Section \ref{sec:applications}, we provide the promised applications of our obtained duality result within the algebraic theory of general De Morgan algebras as well as the bitopological model theory of the S4 modal extension of FDE.

\section{Modal expansions of De Morgan algebras} \label{sec:preliminaries}

In this section, we describe some basics of De Morgan algebras and their S4 modal expansion. For more details on De Morgan algebras, consult \cite{prenosil, pynko}. 
\begin{definition}
    A \emph{De Morgan algebra} is a bounded distributive lattice $\tuple{A;\wedge,\vee,0,1}$ equipped with an involution $-\colon A\to A$ satisfying De Morgan's identities: 
    \begin{enumerate}
        \item $-(a\wedge b)=-a\vee-b$;
        \item $-(a\vee b)=-a\wedge-b$. 
        \end{enumerate}
         \end{definition}
A De Morgan algebra may be equivalently viewed as a bounded distributive lattice $A$ equipped with an order-inverting involution $-\colon A\to A$.

By $\mathbf{DM}$ we denote the variety of De Morgan algebras. Let us note that $\mathbf{DM}$ is a finitely generated variety of algebras. In particular, we have: 
\[\mathbf{DM}=\mathbb{HSP}(\text{DM}_4).\] There is a bijection between the underlying sets of the De Morgan algebra $\text{DM}_4$ and the direct product $\text{B}_2\times \text{B}_2$ where $\text{B}_2$ is the two-element Boolean algebra, however, this bijection is not an isomorphism.
The Hasse diagrams of $\text{B}_2\times \text{B}_2$ and $\text{DM}_4$ are depicted in Figure 1. 

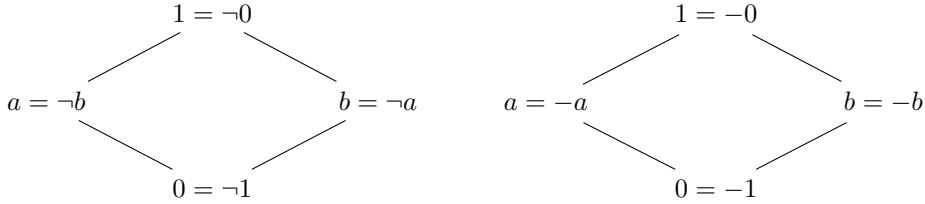
\begin{figure}[ht]
    \centering
\begin{tikzcd}
                             & 1=\neg 0                                         &                              &                          & 1=-0                                         &                          \\
a=\neg b \arrow[ru, no head] &                                                  & b=\neg a \arrow[lu, no head] & a=-a \arrow[ru, no head] &                                              & b=-b \arrow[lu, no head] \\
                             & 0=\neg 1 \arrow[lu, no head] \arrow[ru, no head] &                              &                          & 0=-1 \arrow[ru, no head] \arrow[lu, no head] &                         
\end{tikzcd}
\par
\vspace{.5cm}
    \caption{The $\text{B}_2\times \text{B}_2$ Boolean algebra (left) and $\text{DM}_4$ De Morgan algebra (right)}
    \label{hasse diagrams of de morgan algebras}
    \label{fig:DM4}
\end{figure}

Clearly $\text{B}_2\times \text{B}_2$ has no fixed points with respect to Boolean complementation $\neg$ whereas $\text{DM}_4$ has two fixed points with respect to De Morgan involution. Hence $\text{B}_2\times \text{B}_2$ is not isomorphic to $\text{DM}_4$.

\begin{definition}
    A \emph{modal De Morgan algebra} is a De Morgan algebra $A$ equipped with a normal additive operator $\nabla\colon A \to A$, i.e.:
\begin{enumerate}
    \item $\nabla(a\vee b)=\nabla a\vee\nabla b$; 
    \item $\nabla0=0$. 
\end{enumerate}
\end{definition}
\begin{remark}\label{remark}
    In any modal De Morgan algebra, one may define an operator $\Delta\colon A\to A$ by $\Delta a:=-\nabla-a$ which results in a multiplicative co-normal operator since:
    \begin{align*}
        \Delta(a\wedge b)&=-\nabla-(a\wedge b)\\&=-\nabla(-a\vee-b)\\&=-(\nabla-a\vee\nabla-b)\\&=-\nabla-a\wedge-\nabla-b\\&=\Delta a\wedge\Delta b
    \end{align*}
 and $\Delta 1=-\nabla-1=-\nabla 0=-0=1$ so that $\Delta(a\wedge b)=\Delta a\wedge\Delta b$ and $\Delta 1=1$. 
\end{remark}
\begin{definition}
    An \emph{S4 De Morgan algebra} is a modal De Morgan algebra $\tuple{A;\nabla}$ such that $\nabla$ is increasing and idempotent, i.e.: 
\begin{enumerate}
    \item $a\leq\nabla a$; 
    \item $\nabla\nabla a=\nabla a$. 
\end{enumerate} 
\end{definition}
Clearly if $A$ is an S4 De Morgan algebra, then $\nabla\colon A\to A$ is a closure operator on $A$. 
\begin{example}
    Let $[0,1]$ be the real closed unit interval, let $\eta(a)=1-a$, and let:
    \[c(a)=\begin{cases}
        0, & \text{if $a=0$}\\
        1, & \text{otherwise}
    \end{cases}\]
    Then $\tuple{[0,1];\min,\max,\eta,0,1,c}$ is an S4 De Morgan algebra. 
\end{example}
\begin{remark}
    Clearly if $\tuple{A;\nabla}$ is an S4 De Morgan algebra then $\Delta\colon A\to A$ satisfies $\Delta a\leq a$ and $\Delta\Delta a=\Delta a$. For the former, observe that $-a\leq\nabla-a$ so $-\nabla-a\leq--a$ and hence $-\nabla-a\leq a$ so $\Delta a\leq a$. For the latter, we have: 
    \begin{align*}
        \Delta\Delta a=\Delta-\nabla-a
        =-\nabla--\nabla-a=-\nabla\nabla-a
        =-\nabla-a
        =\Delta a.
    \end{align*}
This, together with Remark \ref{remark} implies that $\Delta$ is an interior operator on $A$. 
\end{remark}

\section{Bitopological representation} \label{sec:representation}

In this section, we provide a bitopological representation theorem by showing that every S4 De Morgan algebra is isomorphic to the $(\tau_1,\delta_2)$-biclopen subsets of a pairwise Stone space. 
\subsection{Pairwise Stone spaces}
A \emph{bitopological space} is a triple $\tuple{X;\tau_1,\tau_2}$ such that $X$ is a non-empty set and $\tau_1,\tau_2\subseteq\wp(X)$ are topologies on $X$. If $\tuple{X;\tau_1,\tau_2}$ and $\langle X';\tau_1',\tau_2'\rangle$ are bitopological spaces, a function $f\colon X\to X'$ is \emph{bicontinuous} if $f\colon \langle X;\tau_1\rangle\to\langle X';\tau_1'\rangle$ and $f\colon \langle X;\tau_2\rangle\to\langle X';\tau_2'\rangle$ are continuous. A bicontinuous function $f\colon X\to X'$ is then said to be a \emph{bihomeomorphism} if $f$ is a bijection with bicontinuous inverse $f^{-1}$. We call a map $f\colon\tuple{X;\tau_1,\tau_2}\to\tuple{X;\tau_1,\tau_2}$ \emph{twist bicontinuous} if $f^{-1}[U]\in\tau_2$ for each $U\in\tau_1$ and $f^{-1}[U]\in\tau_1$ for each $U\in\tau_2$. For a bitopological space $\tuple{X;\tau_1,\tau_2}$, let $\delta_1$ be the collection of closed subsets of $\tuple{X;\tau_1}$ and let $\delta_2$ be the collection of closed subsets of $\tuple{X;\tau_2}$. Moreover, let us set $\beta_1=\tau_1 \cap \delta_2$ and $\beta_2=\tau_2 \cap \delta_1$.
\begin{definition}[\cite{Salbany:BitopologicalSpaces,bezhanishvili}]\label{pairwise stone space}
    Let $\tuple{X;\tau_1,\tau_2}$ be a bitopological space. Then: 
    \begin{enumerate}
        \item $\tuple{X;\tau_1,\tau_2}$ is \emph{pairwise compact} if every cover $\{U_i:i\in I\}$ of $X$ with $U_i\in\tau_1\cup\tau_2$ has a finite subcover; 
        \item $\tuple{X;\tau_1,\tau_2}$ is \emph{pairwise zero-dimensional} if $\beta_1$ is a basis of $\tau_1$ and $\beta_2$ is a basis of $\tau_2$; 
        \item $\tuple{X;\tau_1,\tau_2}$ is \emph{pairwise Hausdorff} if for all $x,y\in X$ such that $x\not=y$, there exist $U\in\tau_1$ and $V\in\tau_2$ such that $U\cap V=\emptyset$ and $x\in U$ and $y\in V$; 
        \item $\tuple{X;\tau_1,\tau_2}$ is a \emph{pairwise Stone space} if it is pairwise compact, pairwise zero-dimensional, and pairwise Hausdorff. 
    \end{enumerate}
\end{definition}
\begin{remark}
Observe that the property of pairwise zero-dimensionality amounts to asserting that open sets in $\tuple{X;\tau_1}$ that are closed in $\tuple{X;\tau_2}$ form a basis for $\tuple{X;\tau_1}$ and that open sets in $\tuple{X;\tau_2}$ that are closed in $\tuple{X;\tau_1}$ form a basis for $\tuple{X;\tau_2}$. 
\end{remark}
If $\tuple{X;\tau_1,\tau_2}$ is a bitopological space, let $\sigma_1$ and $\sigma_2$ denote the collection of compact subsets of $\tuple{X;\tau_1}$ and $\tuple{X;\tau_2}$, respectively. 
 \begin{proposition}\label{p:compact}
    A bitopological space $\tuple{X;\tau_1,\tau_2}$ is pairwise compact if and only if $\delta_1\subseteq\sigma_2$ and $\delta_2\subseteq\sigma_1$. 
\end{proposition}
\begin{proof}
    Consult \cite[Proposition 2.9]{bezhanishvili} for the proof. 
\end{proof}
\begin{definition}
    Let $\tuple{X;\tau_1,\tau_2}$ be a pairwise Stone space. Then a subset $U\subseteq X$ is said to be $(\tau_1,\delta_2)$-\emph{biclopen} if $U\in\tau_1\cap\delta_2$ and $(\delta_1,\tau_2)$-\emph{biclopen} if $U\in\delta_1\cap\tau_2$. 
\end{definition}
\begin{lemma} \label{l:injectivity}
    Let $\langle X,\tau_1,\tau_2\rangle$ be a pairwise zero-dimensional and pairwise Hausdorff space with $x,y \in X$ such that $x \neq y$. Then there is $U \in \beta_1$ and $V \in \beta_2$ such that $U \cap V=\emptyset$ and $x \in U$ and $y \in V$.
\end{lemma}

\begin{proof}
    Since $X$ is pairwise Hausdorff space, it follows there are $\Tilde{U} \in \tau_1$ and $\Tilde{V} \in \tau_2$ such that $x \in \Tilde{U}$, $y \in \Tilde{V}$ and $\Tilde{U} \cap \Tilde{V}=\emptyset$. 
Since $X$ is a pairwise zero-dimensional space, we have
\[\Tilde{U}=\bigcup_{i \in I} U_i,\hspace{.2cm}\Tilde{V}=\bigcup_{j \in J} V_j\]
where $U_i \in \tau_1 \cap \delta_2$ and $V_j \in \tau_2 \cap \delta_1$ for all $i \in I$ and $j \in J$. Let us fix $U \in \tau_1 \cap \delta_2$ with $x \in U \subseteq \Tilde{U}$ and $V \in \tau_2 \cap \delta_1$ such that $y \in V \subseteq \Tilde{V}$. Since $\Tilde{U} \cap \Tilde{V}=\emptyset$, also $U \cap V=\emptyset$.
\end{proof}

\subsection{Pairwise S4 De Morgan Stone spaces}
In this subsection, we introduce the pairwise Stone duals of S4 De Morgan algebras. We make the following observations.

\begin{definition}
For a $T_0$ topological space $\langle X; \tau \rangle$, let $\leq_\tau$ denote the specialization order; that is, for $x, y \in X$, we have $x \leq_\tau y$ if $x$ is in the closure of $\{y\}$. Alternately, the specialization order may be defined by setting $x\leq_{\tau}y$ if and only if $x\in U$ implies $y\in U$ for every $U\in\tau$. For a bitopological space $\langle X; \tau_1, \tau_2 \rangle$ in which both $\langle X; \tau_1 \rangle$ and $\langle X; \tau_2 \rangle$ are $T_0$, let $\leq_{\tau_1}$ and $\leq_{\tau_2}$ denote the specialization orders of $\langle X; \tau_1 \rangle$ and $\langle X; \tau_2 \rangle$, respectively.
\end{definition}

Let us note here that by \cite[Lemma 2.5]{bezhanishvili}, the pairwise zero-dimensionality of $\langle X; \tau_1, \tau_2 \rangle$ implies that the pairwise Hausdorffness of $\langle X; \tau_1, \tau_2 \rangle$ is equivalent to the fact that both $\langle X; \tau_1 \rangle$ and $\langle X; \tau_2 \rangle$ are $T_0$, and thus the specialization orders $\leq_{\tau_1}$ and $\leq_{\tau_2}$ are partial orders.

\begin{lemma}
    Let $\langle X; \tau_1, \tau_2 \rangle$ be a pairwise zero-dimensional and pairwise Hausdorff bitopological space. 
    Then $x \leq_{\tau_1} y$ if and only if $y \leq_{\tau_2} x$.
\end{lemma}

\begin{proof}
    See \cite[Lemma 3.3]{bezhanishvili}.
\end{proof}

\begin{lemma}
 Let $\langle X; \tau_1, \tau_2 \rangle$ be pairwise zero-dimensional and pairwise Hausdorff space and let $g\colon \tuple{X;\tau_1,\tau_2}\to\tuple{X;\tau_1,\tau_2}$ be twist continuous. Then $g$ is order inverting with respect to both $\leq_{\tau_1}$ and $\leq_{\tau_2}$.
\end{lemma}

\begin{proof}

Let $x,y \in X$ with $x \leq_{\tau_1} y$. To show that $g$ is order inverting, we prove $g(y) \leq_{\tau_1} g(x)$. We show that every open set which contains $g(y)$ contains $g(x)$. Let $U \in \tau_1$. Since $g$ is twist continuous, we have $g^{-1}[U] \in \tau_2$. By zero-dimensionality it follows $g^{-1}[U] =\bigcup_{i \in I} U_i$, where $U_i \in \tau_2 \cap \delta_1$. Assume $g(y) \in U$. This means $y \in g^{-1}[U]$. Therefore, there is $i \in I$ such that $y \in U_i$. Since $U_i \in \delta_1$, it follows by definition of the specialization order that $x \in U_i$ as well and thus $x \in  g^{-1}[U]$. Thus $g(x) \in U$ and hence $g(y) \leq_{\tau_1} g(x)$. The argument that $g$ is order inverting with respect to $\leq_{\tau_2}$ runs analogously. 
\end{proof}

If $X$ is a set and $R$ is a binary relation on $X$, let \[R[U]=\{y\in X:xRy\hspace{.1cm}\text{for some $x\in U$}\},\hspace{.2cm}\overline{R}=(X\times X)\setminus R\] so that $R[U]$ is the image of $U$ under $R$ and $\overline{R}$ is the complement of $R$. 

\begin{definition} \label{pairwise S4 De Morgan space}
    A \emph{pairwise S4 De Morgan Stone space} (\emph{PS4D-space}) is a relational bitopological space $\langle X;g,R,\tau_1,\tau_2\rangle$ satisfying the following conditions: 
    \begin{enumerate}
        \item $\tuple{X;\tau_1,\tau_2}$ is a pairwise Stone space;   
        \item $g\colon \langle X; \tau_1, \tau_2 \rangle\to \langle X; \tau_1,\tau_2 \rangle$ is a twist continuous involution;   
        \item $R\subseteq X\times X$ is reflexive and transitive; 
        \item if $U\in\tau_1\cap\delta_2$, then $R[U]\in\tau_1\cap\delta_2$;
        \item if $x\overline{R}y$, there exists $U\in\tau_1\cap\delta_2$ such that $x\in U$ and $y\not\in R[U]$.
    \end{enumerate}
\end{definition} 


\begin{remark}
Let us note that a twist continuous involution on a space $\langle X, \tau_1, \tau_2 \rangle$ is actually a homeomorphism between $\langle X, \tau_1 \rangle$ and $\langle X, \tau_2 \rangle$. This easily follows as any involution is its own inverse, and consequently every continuous involution has a continuous inverse.
\end{remark}

Recall that if $A$ is a bounded lattice, a non-empty subset $x\subseteq A$ is a \emph{filter} provided $x$ is upward closed and closed under finite meets. Moreover, we say that a filter $x$ is \emph{proper} if $x\not=A$, i.e., $0\not\in x$, and \emph{prime} if $x$ is a proper filter such that $a\vee b\in x$ implies $a\in x$ or $b\in x$ for all $a,b\in A$. Observe that for any $a\in A$, we have ${\uparrow}a=\{b\in A:a\leq b\}$ is a filter, known as the \emph{principal filter} generated by $a$. Dually, a non-empty set $x\subseteq A$ is an \emph{ideal} provided $x$ is downwards closed and closed under finite joins. Moreover, we say that an ideal $x$ is \emph{proper} if $x\not=A$, i.e., $1\not\in x$, and \emph{prime} if $x$ is a proper ideal such that $a\wedge b\in x$ implies $a\in x$ or $b\in x$ for all $a,b\in A$. Observe that for any $a\in A$, we have ${\downarrow}a=\{b\in A:b\leq a\}$ is an ideal, known as the \emph{principal ideal} generated by $a$. 

The following Prime Filter Theorem is well-known and will be exploited later on. 
 \begin{theorem}\label{prime filter theorem}
     Let $A$ be a distributive lattice, let $x\subseteq A$ be a filter, and let $y\subseteq A$ be an ideal such that $x\cap y=\emptyset$. Then there exists a prime filter $x_\rho\subseteq A$ such that $x\subseteq x_\rho$ with $y\cap x_\rho=\emptyset$. 
 \end{theorem}
 \begin{proof}
 Consult Gr\"atzer \cite[pg. 84]{gratzer} for a proof. 
 \end{proof}
\begin{definition}
    Let $A$ be an S4 De Morgan algebra. The \emph{bitopological spectrum} of $A$ is a relational bitopological space $\mathcal{S}_0(A)=\langle\mathfrak{P}(A);g_A;R_A;\tau_+(\beta_+),\tau_-(\beta_-)\rangle$ such that: 
    \begin{enumerate}
        \item $\mathfrak{P}(A)$ is the collection of all prime filters of $A$; 
        \item $g_A\colon\mathfrak{P}(A)\to\mathfrak{P}(A)$ is defined by $g_A(x)=\{a\in A:-a\not\in x\}$; 
        \item $R_A\subseteq\mathfrak{P}(A)\times\mathfrak{P}(A)$ is defined by $xR_Ay$ iff $a\in x$ implies $\nabla a\in y$; 
        \item $\tau_+(\beta_+)$ is the topology on $\mathfrak{P}(A)$ generated by the basis: \[\beta_+=\{\phi_+(a) : a \in A\}\hspace{.2cm}\text{where}\hspace{.2cm}\phi_+(a)=\{x\in\mathfrak{P}(A):a\in x\};\]
        \item $\tau_-(\beta_-)$ is the topology on $\mathfrak{P}(A)$ generated by the basis: 
        \[\beta_-=\{\phi_-(a) : a \in A\}\hspace{.2cm}\text{where}\hspace{.2cm}\phi_-(a)=\{x\in\mathfrak{P}(A):a\not\in x\};\]
    \end{enumerate}
\end{definition}

\begin{lemma}
    If $A$ is an S4 De Morgan algebra, then: 
    \begin{enumerate}
        \item $\phi_-(a\wedge b)=\phi_-(a)\cup\phi_-(b)$; 
        \item $\phi_-(a\vee b)=\phi_-(a)\cap\phi_-(b)$; 
        \item $\phi_-(0)=\mathfrak{P}(A)$ and $\phi_-(1)=\emptyset$. 
    \end{enumerate}
    \end{lemma}
    \begin{proof}
        The calculations of conditions 1 and 2 are straightforward and proceed as follows: 
        \begin{align*}
            \phi_-(a\wedge b)&=\{x\in\mathfrak{P}(A):a\wedge b\not\in x\}
            \\&=\{x\in\mathfrak{P}(A):a\not\in x\hspace{.2cm}\text{or}\hspace{.2cm}b\not\in x\}
            \\&=\{x\in\mathfrak{P}(A):a\not\in x\}\cup\{x\in\mathfrak{P}(A):b\not\in x\}
            \\&=\phi_-(a)\cup\phi_-(b);
        \end{align*}
\begin{align*}
    \phi_-(a\vee b)&=\{x\in\mathfrak{P}(A):a\vee b\not\in x\}
\\&=\{x\in\mathfrak{P}(A):a\not\in x\hspace{.2cm}\text{and}\hspace{.2cm}b\not\in x\}
 \\&=\{x\in\mathfrak{P}(A):a\not\in x\}\cap\{x\in\mathfrak{P}(A):b\not\in x\}
 \\&=\phi_-(a)\cap\phi_-(b).
\end{align*}  
Moreover, since $\mathfrak{P}(A)$ is the collection of all prime filters of $A$, each $x\in\mathfrak{P}(A)$ is proper and hence it is obvious that $\phi_-(0)=\{x\in\mathfrak{P}(A):0\not\in x\}=\mathfrak{P}(A)$ and since each $x\in\mathfrak{P}(A)$ is an upset, $\phi_-(1)=\{x\in\mathfrak{P}(A):1\not\in x\}=\emptyset$, which completes the proof. 
    \end{proof}

We first show that the bitopological spectrum of any S4 De Morgan algebra gives rise to a PS4D-space. 
\begin{lemma}\label{algebra to space}
    If $A$ is an S4 De Morgan algebra, then $\mathcal{S}_0(A)$ is a PS4D-space. 
\end{lemma}
\begin{proof}
   We first demonstrate that $\langle\mathfrak{P}(A);\tau_+(\beta_+),\tau_-(\beta_-)\rangle$ forms a pairwise Stone space. The proof is analogous to \cite[Proposition 5.1]{bezhanishvili} however we outline the steps for the sake of clarity. To see that $\langle\mathfrak{P}(A);\tau_+(\beta_+),\tau_-(\beta_-)\rangle$ is a pairwise Hausdorff space, take any $x,y\in\mathfrak{P}(A)$ such that $x\not=y$. Without loss of generality, we may assume $x\not\subseteq y$ so that there exists $a\in A$ such that $a\in x$ but $a\not\in y$. Then $x\in\phi_+(a)$ and $y\in\phi_-(a)$ with $\phi_+(a)\in\tau_+(\beta_+)$ and $\phi_-(a)\in\tau_-(\beta_-)$ such that $\phi_+(a)\cap\phi_-(a)=\emptyset$.

   For compactness, we show that every cover of $X$ by elements of $\beta_+\cup\beta_-$ has a finite cover. Let: \[X=\bigcup\{\phi_+(a_i):i\in I\}\cup\bigcup\{\phi_-(b_k):k\in K\}\] for some $a_i,b_k\in A$. Let $x$ be the ideal generated by $\{a_i:i\in I\}$ and let $y$ be the filter generated by $\{b_k:k\in K\}$. Since it can be easily shown by way of the Prime Filter Theorem that $x\cap y\not=\emptyset$, one can find $a_{i_1},\dots a_{i_m}$ and $b_{k_1}, \dots, b_{k_n}$ satisfying $\bigwedge^n_{j=1} b_{k_j}\leq\bigvee^m_{j=1} a_{i_j}$ and hence $\bigcap^n_{j=1} \phi_-(b_{k_j})\subseteq\bigcup^m_{j=1} \phi_+(a_{i_j})$ which implies:
   \[\bigcup^n_{j=1} \phi_-(b_{k_j})\cup\bigcup^m_{j=1} \phi_+(a_{i_j})=X.\] Therefore, $\{\phi_+(a_{i_1}), \dots,\phi_+(a_{i_m}), \phi_-(b_{k_1}),\dots,\phi_-(b_{k_n})\}$ is a finite subcover of $\{\phi_+(a_i):i\in I\}\cup\{\phi_-(b_k):k\in K\}$ which implies that $\langle\mathfrak{P}(A);\tau_+(\beta_+),\tau_-(\beta_-)\rangle$ is pairwise compact.

   To see that $\langle\mathfrak{P}(A);\tau_+(\beta_+),\tau_-(\beta_-)\rangle$ is pairwise zero dimensional, it suffices to demonstrate that $\beta_+=\tau_+\cap\delta_-$. Hence take any $U\in\beta_+$. Then $U\in\tau_+$ and hence $U=\phi(a)$ for some $a\in A$. Since $\phi_+(a)=\mathfrak{P}(A)\setminus\phi_-(a)$ we have $X\setminus U=\phi_-(a)$ so $X\setminus U\in\beta_-$, which implies $U\in\delta_-$ so $U\in\tau_+\cap\delta_-$. The converse inclusion is obtained by compactness together with Proposition \ref{p:compact}. We note that it can be similarly demonstrated that $\beta_-=\tau_-\cap\delta_+$. Hence we conclude that $\langle\mathfrak{P}(A);\tau_+(\beta_+),\tau_-(\beta_-)\rangle$ is pairwise zero-dimensional. Since $\mathcal{S}_0(A)$ is a pairwise Stone space, condition 1 of Definition \ref{pairwise S4 De Morgan space} is satisfied.

    To see that condition 2 of Definition \ref{pairwise S4 De Morgan space} is satisfied, we first check that $g_A$ is well-defined, i.e., that $g_A(x)\in\mathfrak{P}(A)$ for any $x\in\mathfrak{P}(A)$. Hence take any $a\in g_A(x)$ and suppose $a\leq b$. The former gives $-a\not\in x$ and the latter implies $-b\leq-a$. Thus we have $-b\not\in x$ since $x$ is upwards closed and hence $b\in g_A(x)$ so $g_A(x)$ is upward closed for any $x\in\mathfrak{P}(A)$. To see that $g_A(x)$ is closed under finite meets, let $a,b\in g_A(x)$ but $a\wedge b\not\in g_A(x)$. The latter hypothesis together with De Morgan's identities yield $-(a\wedge b)=-a\vee -b\in x$ but the former hypothesis gives $-a\not\in x$ and $-b\not\in x$. Since $x$ is a prime filter, we arrive at a contradiction. Therefore $g_A(x)$ is closed under finite meets for every $x\in\mathfrak{P}(A)$ and is hence a filter. It is clearly proper since otherwise, if $0\in g_A(x)$, then $-0=1\not\in x$, which contradicts the fact that $x$ is upward closed. Hence it remains to verify that $g_A(x)$ is prime. Hence assume $a\vee b\in g_A(x)$ so that $-(a\vee b)=-a\wedge-b\not\in x$. If both $a\not\in g_A(x)$ and $b\not\in g_A(x)$, then $-a,-b\in x$ and since $x$ is closed under meets, we have $-a\wedge-b\in x$, a contradiction. Hence $g_A(x)$ is a prime filter for each $x\in\mathfrak{P}(A)$. Thus $g_A$ is well-defined. We now check that $g_A$ is twist continuous. First note that for any $a\in A$, the definition of $g_A^{-1}$ and the definition of $\phi_+$ yield:  
    \[g_A^{-1}[\phi_+(a)]=\{x\in\mathfrak{P}(A):g_A(x)\in\phi_+(a)\}=\{x\in\mathfrak{P}(A):a\in g_A(x)\}.\]
Then the definitions of $g_A$ and $\phi_-$ then provide the following
\[\{x\in\mathfrak{P}(A):a\in g_A(x)\}=\{x\in\mathfrak{P}(A):-a\not\in x\}=\phi_-(-a)\] where $\phi_-(-a)\in\tau_-(\beta_-)$ and $\phi_+(a)\in\tau_+(\beta_+)$. Similarly, we have 
    \[g_A^{-1}[\phi_-(a)]=\{x\in\mathfrak{P}(A):g_A(x)\in\phi_-(a)\}=\{x\in\mathfrak{P}(A):a\not\in g_A(x)\}.\] Then the definitions of $g_A$ and $\phi_+$ yield 
\[\{x\in\mathfrak{P}(A):a\not\in g_A(x)\}=\{x\in\mathfrak{P}(A):-a\in x\}=\phi_+(-a)\] where $\phi_-(a)\in\tau_-(\beta_-)$ and $\phi_+(-a)\in\tau_+(\beta_+)$. To see that it is an involution, observe that repeated applications of $g_A$ yield the following: 
    \[g_A(g_A(x))=\{a\in A:-a\not\in g_A(x)\}=\{a\in A:--a\in x\}.\] Then by applying the fact $-\colon A\to A$ is an involution, we obtain: 
    \[\{a\in A:--a\in x\}=\{a\in A:a\in x\}=x.\]
Therefore it follows that $g_A$ is a twist continuous involution on $\mathfrak{P}(A)$ and hence condition 2 of Definition \ref{pairwise S4 De Morgan space} is satisfied.

To see that $R_A$ is reflexive, observe that $a\leq\nabla a$ for all $a\in A$ and hence if $a\in x$, then $\nabla a\in x$ so $xR_Ax$ by the definition of $R_A$. To see that $R_A$ is transitive, assume $xR_Ay$ and $yR_Az$. If $a\in x$, then $\nabla a\in y$ since $xR_Az$ and hence $\nabla\nabla a\in z$ since $yR_Az$. However, we have $\nabla\nabla a=\nabla a$ since $\nabla$ is idempotent and hence $\nabla a\in z$ so $xR_Az$, as required. Therefore condition 3 of Definition \ref{pairwise S4 De Morgan space} is satisfied.

   For condition 4, we verify the stronger claim that $\phi_+(\nabla a)=R_A[\phi_+(a)]$. Assume $x\in R_A[\phi_+(a)]$ so that $yR_Ax$ for some $y\in\phi_+(a)$ and hence $a\in y$ and thus $\nabla a\in x$ by the definition of $R_A$. Therefore we have $x\in\phi_+(\nabla a)$ and hence $R_A[\phi_+(a)]\subseteq\phi_+(\nabla a)$. For the other inclusion, suppose $x\in\phi_+(\nabla a)$ so that $\nabla a\in x$. Since $\nabla$ is idempotent, our hypothesis yields $\nabla\nabla a\in x$ and hence $xR_Ax$ with $x\in\phi_+(a)$ so $x\in R_A[\phi_+(a)]$ and thus $\phi_+(\nabla a)\subseteq R_A[\phi_+(a)]$ so $\phi_+(\nabla a)=R_A[\phi_+(a)]$.

   Therefore, since $\phi_+(\nabla a)\in\tau_+$, we have $R_A[\phi_+(a)]\in\tau_+$ and since $\phi_+(a)=\mathfrak{P}(A)\setminus\phi_-(a)$, we have $\phi_+(\nabla a)=\mathfrak{P}(A)\setminus\phi_-(\nabla a)\in\delta_-$ so we conclude that $R_A[\phi_+(a)]\in\tau_+\cap\delta_-$.

   Finally, for condition 5, assume $x\overline{R_A}y$ so that there exists some $a\in A$ such that $a\in x$ but $\nabla a\not\in y$. Then $x\in\phi_+(a)$ and $y\not\in\phi_+(\nabla a)=R_A[\phi_+(a)]$ with $\phi_+(a),R_A[\phi_+(a)]\in\tau_+\cap\delta_-$. Thus we conclude that $\mathcal{S}_0(A)$ is a PS4D-space.        \end{proof}

We must now conversely show that the algebra of $(\tau_1,\delta_2)$-biclopen subsets of a PS4D-space induce an S4 De Morgan algebra.  

\begin{lemma}\label{space to algebra}
    If $X$ is a PS4D-space, then $\mathcal{A}_0(X)=\langle\tau_1\cap\delta_2;\cap,\cup,^*,\emptyset,X,\nabla_R\rangle$ is an S4 De Morgan algebra under $U^*=\{x\in X:g(x)\not\in U\}$ and $\nabla_RU=R[U]$ for $U\in\tau_1\cap\delta_2$.    \end{lemma}
\begin{proof}
    Obviously $\emptyset,X\in \tau_1\cap\delta_2$ and $U\cap V$, $U\cup V\in\tau_1\cap\delta_2$ if $U,V\in\tau_1\cap\delta_2$ and by condition 3 of Definition \ref{pairwise S4 De Morgan space} we have $R[U]=\nabla_R U\in\tau_1\cap\delta_2$ if $U\in\tau_1\cap\delta_2$. To see $U^*\in\tau_1\cap\delta_2$, we first note that $U^*=\{x\in X:g(x)\not\in U\}=X\setminus g^{-1}[U]$. Let $U\in\tau_1\cap\delta_2$, so that $U\in\tau_1$ and $U\in\delta_2$. Since $g$ is twist continuous, we have $g^{-1}[U]\in\tau_2$ so $X\setminus g^{-1}[U]\in\delta_2$. Similarly, we have $g^{-1}[U]\in \delta_1$ so $X\setminus g^{-1}[U]\in\tau_1$ and hence $U^*=X\setminus g^{-1}[U]\in\tau_1\cap\delta_2$. Hence the operations associated with $\mathcal{A}_0(X)$ are well-defined. It is easy to see that $\langle\tau_1\cap\delta_2;\cap,\cup,\emptyset,X\rangle$ is a bounded distributive lattice and hence it remains to verify that $^*\colon\tau_1\cap\delta_2\to\tau_1\cap\delta_2$ is an order-inverting involution with respect to subset inclusion and that $\nabla_R\colon\tau_1\cap\delta_2\to\tau_1\cap\delta_2$ is a closure operator. To see that $^*$ is an order-inverting involution, first let $U,V\in\tau_1\cap\delta_2$ be such that $U\subseteq V$. If $x\in V^*$, then $g(x)\not\in V$ and hence $g(x)\not\in U$. This implies that $x\in U^*$ and thus $V^*\subseteq U^*$. Now assume $x\in U^{**}$, then $g(x)\not\in U^*$ and hence $g(g(x))\in U$. Since $g(g(x))=x$ for all $x\in X$, we have $x\in U$ so $U^{**}\subseteq U$. Conversely, if $x\in U$ but $x\not\in U^{**}$, then $g(x)\in U^*$ so $g(g(x))\not\in U$ but $g(g(x))=x$ so $x\not\in U$, which contradicts our hypothesis.

The result that $\nabla_R$ is normal is obvious since $\nabla_R
\emptyset=R[\emptyset]=\emptyset$. To see that $\nabla_R$ is additive, assume $x\in \nabla_R(U\cup V)=R[U\cup V]$. Then $yRx$ for some $y\in U\cup V$. If $y\in U$, then $x\in R[U]=\nabla_RU$ so $x\in\nabla_R U\cup\nabla_RV$. If $y\in V$, then $x\in R[V]=\nabla_RV$ so $x\in \nabla_RU\cup\nabla_RV$. Hence we have $\nabla_R(U\cup V)\subseteq\nabla_R U\cup\nabla_RV$. Conversely, suppose $x\in\nabla_RU\cup\nabla_RV=R[U]\cup R[V]$. If $x\in R[U]$, then $yRx$ for some $y\in U$ so $y\in U\cup V$ and hence $x\in R[U\cup V]=\nabla_R(U\cup V)$. If $x\in R[V]$, then $yRx$ for some $y\in V$ so $y\in U\cup V$ and hence $x\in R[U\cup V]=\nabla_R(U\cup V)$. Hence we conclude $\nabla_R(U\cup V)\supseteq\nabla_RU\cup\nabla_RV$ and therefore $\nabla_R(U\cup V)=\nabla_RU\cup\nabla_RV$.    
    
    To verify that $\nabla_R$ is increasing, suppose $x\in U$. Then the reflexivity of $R$ yields $xRx$ and hence $x\in R[U]=\nabla_RU$ so $U\subseteq\nabla_RU$. This inclusion implies $\nabla_RU\subseteq\nabla_R\nabla_RU$ and hence we verify $\nabla_R\nabla_RU\subseteq\nabla_RU$ to also see that $\nabla_R$ is idempotent. If $x\in \nabla_R\nabla_RU=R[R[U]]$, then $yRx$ for some $y\in R[U]$ where $zRy$ for some $z\in U$. Since $R$ is transitive, we have $zRx$ with $z\in U$ and hence $x\in R[U]=\nabla_RU$.      
\end{proof}

\begin{theorem}\label{rep thm}
    Every S4 De Morgan algebra $A$ is isomorphic to $\mathcal{A}_0(\mathcal{S}_0(A))$. 
\end{theorem}
\begin{proof}
  We show that $\phi_+$ exhibits an isomorphism from $A$ to $\mathcal{A}_0(\mathcal{S}_0(A))$. For injectivity, take $a,b\in A$ and suppose $a\not=b$. Without loss of generality, suppose $a\not\leq b$. Then consider the principal filter generated by $a$, i.e., ${\uparrow}a=\{c\in A:a\leq c\}$ and the principal ideal generated by $b$, i.e., ${\downarrow}b=\{c\in A:c\leq b\}$. Since clearly we have ${\uparrow}a\cap{\downarrow}b=\emptyset$, it follows by the Prime Filter Theorem i.e., Theorem \ref{prime filter theorem}, that there exists a prime filter $x_\rho\in\mathfrak{P}(A)$ such that ${\uparrow}a\subseteq x_\rho$ with $x_\rho\cap{\downarrow}b=\emptyset$. Since $a\in x_\rho$ and $b\not\in x_\rho$, we have $x_\rho\in\phi_+(a)$ but $x_\rho\not\in\phi_+(b)$ and hence $\phi_+(a)\not\subseteq\phi_+(b)$ so $\phi_+(a)\not=\phi_+(b)$. Now observe that: 
  \begin{align*}
      \phi_+(-a)&=\{x\in\mathfrak{P}(A):-a\in x\}=\{x\in\mathfrak{P}(A):a\not\in g_A(x)\}\\&=\{x\in\mathfrak{P}(A):g_A(x)\not\in\phi_+(a)\}=\phi_+(a)^*.
  \end{align*}
  Recall from the proof of Lemma \ref{algebra to space} that $\phi_+(\nabla a)=R_A[\phi_+(a)]$ and since $\nabla_{R_A}\phi_+(a)=R_A[\phi_+(a)]$ by our construction of $\mathcal{A}_0(\mathcal{S}_0(A))$, we have $\phi_+(\nabla a)=\nabla_{R_A}\phi_+(a)$. Thus $\phi_+$ is a homomorphism for $-$ and $\nabla$. The proof that $\phi_+$ is a homomorphism for meets is standard: 
\begin{align*}
    \phi_+(a\wedge b)&=\{x\in\mathfrak{P}(A):a\wedge b\in x\}\\&=\{x\in\mathfrak{P}(A):a\in x\}\cap\{x\in\mathfrak{P}(A):b\in x\}\\&=\phi_+(a)\cap\phi_+(b).
\end{align*}
 The proof that $\phi_+$ is a homomorphism for joins runs dually to the case of meets:
\begin{align*}
    \phi_+(a\vee b)&=\{x\in\mathfrak{P}(A):a\vee b\in x\}\\&=\{x\in\mathfrak{P}(A):a\in x\}\cup\{x\in\mathfrak{P}(A):b\in x\}\\&=\phi_+(a)\cup\phi_+(b).
\end{align*}
Clearly $\phi_+(0)=\{x\in\mathfrak{P}(A):0\in x\}=\emptyset$ since each $x\in\mathfrak{P}(A)$ is a prime filter and hence proper. Moreover, we have $\phi_+(1)=\{x\in\mathfrak{P}(A):1\in x\}=\mathfrak{P}(A)$ since prime filters (and indeed, arbitrary filters) are upward closed. Therefore $\phi_+$ is a homomorphic embedding from $A$ to $\mathcal{A}_0(\mathcal{S}_0(A))$. Lastly, note that by Lemma \ref{algebra to space}, each $(\tau_1,\delta_2)$-biclopen set in $\mathcal{S}_0(A)$ is of the form $\phi_+(a)$ for some $a\in A$ and hence $\phi_+$ is a surjection and thus $\phi_+$ is an isomorphism.      
\end{proof}

\section{Duality} \label{sec:duality}
 
\subsection{Algebraic Realization of PS4D-Spaces}
In this subsection, we provide an algebraic realization result by showing that every PS4D-space $X$ is homeomorphic and relationally isomorphic to the prime filter spectrum $\mathcal{S}_0(\mathcal{A}_0(X))$ of the algebra $\mathcal{A}_0(X)$ of $(\tau_1,\delta_2)$-biclopen subsets of $X$.   

The following well-known result in topology will be used in the proof of Theorem \ref{algebraic realization}. 

\begin{lemma}\label{topology lemma}
Let $X$ be a compact Hausdorff space and let $X'$ be a Hausdorff space. If $f\colon X\to X'$ is a continuous bijection, then $f$ is a homeomorphism. 
\end{lemma}
\begin{proof}
See Davey and Priestley \cite[Lemma A.7, p. 277]{davey}. 
\end{proof}

\begin{theorem}\label{algebraic realization}
    Every PS4D-space $X$ is homeomorphic to $\mathcal{S}_0(\mathcal{A}_0(X))$.
\end{theorem}
\begin{proof}
We show that $\psi(x)=\{U\in\tau_1\cap\delta_2:x\in U\}$ exhibits the desired homeomorphism. We first show that $\psi$ is well-defined by demonstrating that $\psi(x)$ is a prime filter in $\mathcal{A}_0(X)$ for any $x\in X$. We will assume $U,V \in \tau_1 \cap \delta_2$. We will show that $\psi(x)$ is a filter. Let  us assume $U,V \in \psi(x)$. By definition $x \in U$ and $x \in V$ hence $x \in U \cap V$ and thus $U \cap V \in \psi(x)$. Therefore, $\psi$ is closed under finite meets. Now let $U \in \psi(x)$ and $U \subseteq V$. Since $x \in U$ also $x \in V$ and thus $V \in \psi(x)$. Hence, $\psi$ is also closed under upsets and thus $\psi(x)$ is a filter. To show that $\psi(x)$ is prime we assume $U \cup V \in \psi(x)$. Since $x \in U \cup V$ it follows that $x \in U$ or $x \in V$. Therefore, $U \in \psi(x)$ or $V \in \psi(x)$. This shows that $\psi(x)$ is prime.

To show that $\psi$ is injective we assume $x,y \in X$ with $x\not=y$.
    Since $X$ is pairwise zero-dimensional and pairwise Hausdorff, it follows by Lemma \ref{l:injectivity} there is $U \in \tau_1 \cap \delta_2$ and $V \in \tau_2 \cap \delta_1$ such that $x \in U$ and $y \notin U$. Thus, $U \in \psi(x)$, $U  \notin \psi(y)$ and hence $\psi(x) \neq \psi(y)$.

To see that $\psi$ is surjective, let $F$ be a prime filter of $\beta_1$ and let $G=\{H\in\beta_2:X\setminus H\not\in F\}$. It is easy to see that $G$ is a prime filter of $\beta_2$ and that $F\cup G$ satisfies the finite intersection property. Since $X$ is a pairwise compact space as well as a pairwise Hausdorff space, there exists some $x\in X$ satisfying $\bigcap(F\cup G)=\{x\}$, and hence $\psi(x)=F$ so $\psi$ is surjective as $F$ was chosen arbitrarily.  

It remains to prove continuity. Let us recall that for $U\in\mathcal{A}_0(X)$, sets of the form $\varphi_+(U)$ are basic open in $\tau_+$ and sets of the form $\varphi_-(U)$ are basic open in $\tau_-$. We have: 
$$\psi^{-1}[\varphi_+(U)]=\{x \in X : \psi(x) \in \varphi_+(U)\}=\{x \in X : U \in \psi(x)\}=\{x \in X : x \in U\}=U$$
where $U\in\tau_1$. Analogously, we have: 
\[\psi^{-1}[\phi_-(U)]=\{x\in X:\psi(x)\in \phi_-(U)\}=\{x\in X:U\not\in\psi(x)\}=\{x\in X:x\not\in U\}=X\setminus U\] where $X\setminus U\in\tau_2$. It is easy to see that by Lemma \ref{topology lemma}, if $Y$ is a pairwise compact and pairwise Hausdorff space, $Z$ is a pairwise Hausdorff space, and $f\colon Y\to Z$ is a bicontinuous bijection, then $f$ is a bihomeomorphism. Thus since $X$ is a pairwise compact and pairwise Hausdorff, $\mathcal{S}_0(\mathcal{A}_0(X))$ is pairwise Hausdorff, and $\psi$ has been shown to be a bicontinuous bijection from $X$ to $\mathcal{S}_0(\mathcal{A}_0(X))$, it follows that $\psi$ is a bihomeomorphism. 
\end{proof}

\begin{theorem}\label{relational isomorphism}
    Let $X$ be a PS4D-space. Then: 
    \begin{enumerate}
        \item $xRy\Longleftrightarrow\psi(x)R\psi(y)$; 
        \item $\psi(g(x))=g(\psi(x))$. 
    \end{enumerate}
\end{theorem}
\begin{proof}

If $x\overline{R}y$ holds in $X$ then by Definition \ref{pairwise S4 De Morgan space}(5) it follows there is $U \in \tau_1 \cap \delta_2$ such that $x \in U$ and $y \notin R[U]$. Alternatively, this can be rewritten as there is $U \in \tau_1 \cap \delta_2$ such that $U \in \psi(x)$ and $\nabla_R U \notin \psi(y)$
, which implies that $\psi(x)\overline{R}\psi(y)$. Conversely, if $\psi(x)\overline{R}\psi(y)$, then there exists some $U \in \psi(x)$ such that $\nabla_R U \notin \psi(y)$, meaning $x \in U$ and $y \notin R[U]$. Thus $x\overline{R}y$.
Therefore, we showed that $x\overline{R}y$ in $X$ if and only if $\psi(x)\overline{R}\psi(y)$ holds in $\mathcal{S}_0(\mathcal{A}_0(X))$.

Also, by the definition of $g$ and $U^*$ it follows
\begin{align*}
g(\psi(x))&=\{U \in \mathcal{A}_0(X): U^*\notin \psi(x)\}=\{U \in \mathcal{A}_0(X): x\notin U^*\}\\&=\{U \in \mathcal{A}_0(X): g(x)\in U\}=\psi(g(x)).
\end{align*}
This completes the proof. 
\end{proof}

\subsection{The main result}
In this subsection, we establish bitopological duality for the category of S4 De Morgan algebras. 
\begin{definition} \label{De Morgan homomorphism}
    Let $A$ and $A'$ be S4 De Morgan algebras. A function $h\colon A\to A'$ is a \emph{homomorphism} provided: 
    \begin{enumerate}
        \item $h$ is a bounded lattice homomorphism; 
        \item $h(-a)=-h(a)$ and $h(\nabla a)=\nabla h(a)$. 
    \end{enumerate}
\end{definition}

\begin{definition}\label{parwise frame homomorphism}
    Let $X= \langle X, g, R, \tau_1, \tau_2 \rangle$ and $X' = \langle X', g', R', \tau'_1, \tau'_2 \rangle$ be PS4D-spaces. \\ A function $f\colon X\to X'$ is a \emph{pairwise frame morphism} provided:
    \begin{enumerate}
        \item $f$ is a bicontinuous function; 
        \item $f(g(x))=g'(f(x))$ for all $x\in X$; 
        \item $f^{-1}[R'[U]]=R[f^{-1}[U]]$ for all $U\in\tau'_1\cap\delta'_2$. 
    \end{enumerate}
\end{definition}
\begin{definition}
    By $\mathbf{S4D}$ we denote the category of S4 De Morgan algebras and homomorphisms. By $\mathbf{PStone_{S4D}}$, we denote the category of pairwise S4 De Morgan Stone spaces and bicontinuous frame morphisms.  
\end{definition}
\begin{lemma}\label{cont to homo}
    Let $X= \langle X, g, R, \tau_1, \tau_2 \rangle$ and $X' = \langle X', g', R', \tau'_1, \tau'_2 \rangle$ be PS4D-spaces and let $f\colon X\to X'$ be a bicontinuous frame morphism. Then $\mathcal{A}_1(f)\colon\mathcal{A}_0(X')\to\mathcal{A}_0(X)$ is a homomorphism of S4 De Morgan algebras under $\mathcal{A}_1(f)=f^{-1}$.  
\end{lemma}

\begin{proof}
Standard set-theoretic arguments show  \[\mathcal{A}_1(f)[U \cap V]=\mathcal{A}_1(f)[U] \cap \mathcal{A}_1(f)[V],\hspace{.2cm}\mathcal{A}_1(f)[U \cup V]=\mathcal{A}_1(f)[U] \cup \mathcal{A}_1(f)[V].\] Next, we have 
$$\mathcal{A}_1(f)[U^*]=f^{-1}[U^*]=\{x \in X : f(x) \in U^*\}=\{x \in X : g'(f(x)) \notin U\}.$$
By Definition \ref{parwise frame homomorphism}(2) this is equal to
$$\{x \in X : f(g(x)) \notin U\}=\{x \in X : g(x) \notin f^{-1}[U]\}=(f^{-1}[U])^*=(\mathcal{A}_1(f)[U])^*.$$
Clearly, we have $\mathcal{A}_1(X')=X$ and $\mathcal{A}_1(\emptyset)=\emptyset$. Lastly, using Definition \ref{parwise frame homomorphism}(3) we have 
$$\mathcal{A}_1(f)[\nabla'_R U]=f^{-1}[R'[U]]=R[f^{-1}[U]]=\nabla_R \mathcal{A}_1(f)[U].$$  Therefore, we conclude that $\mathcal{A}_1(f)$ is a homomorphism of S4 De Morgan algebras. 
\end{proof}

\begin{lemma}\label{homo to cont}
    Let $A$ and $B$ be S4 De Morgan algebras and let $h:A \to B$ be a homomorphism. Then $\mathcal{S}_1(h):\mathcal{S}_0(B) \to \mathcal{S}_0(A)$ defined by $\mathcal{S}_1(h)=h^{-1}$ is a bicontinuous frame morphism.
\end{lemma}

\begin{proof}
We start by showing that $\mathcal{S}_1(h)$ is bicontinuous. It is enough to verify this on the basis $\{\varphi_+(a):a\in A\}$ and $\{\varphi_-(a):a\in A\}$.
We have the following.
\begin{align*}
    \mathcal{S}_1(h)^{-1}[\varphi_+(a)]&=(h^{-1})^{-1}[\varphi_+(a)]=\{x \in \mathcal{S}_0(B) : h^{-1}[x] \in \varphi_+(a)\}\\&=\{x \in \mathcal{S}_0(B) : a \in h^{-1}[x]\}=\{x \in \mathcal{S}_0(B) : h(a) \in x\}=\varphi_+(h(a)).
\end{align*}
Similarly, we have
\begin{align*}
    \mathcal{S}_1(h)^{-1}[\varphi_-(a)]&=(h^{-1})^{-1}[\varphi_-(a)]=\{x \in \mathcal{S}_0(B) : h^{-1}[x] \in \varphi_-(a)\}\\&=\{x \in \mathcal{S}_0(B) : a \notin h^{-1}[x]\}=\{x \in \mathcal{S}_0(B) : h(a) \notin x\}=\varphi_-(h(a)).
\end{align*}
It remains to check the conditions (2) and (3) from Definition  \ref{parwise frame homomorphism}.
To show $\mathcal{S}_1(h)(g_B(x))=g_A(\mathcal{S}_1(h)(x))$ we have
$$\mathcal{S}_1(h)(g_B(x))=\mathcal{S}_1(h)[\{b \in B : -b \notin x\}]=h^{-1}[\{b \in B : -b \notin x\}]=\{a \in A : -h(a) \notin x\}.$$ Using Definition \ref{De Morgan homomorphism} (2) this is equal to 
$$\{a \in A : h(-a) \notin x\}=\{a \in A : -a \notin h^{-1}[x]\}=g_A(h^{-1}[x])=g_A(\mathcal{S}_{1}(h)(x)).$$

It remains to show $\mathcal{S}_1(h)^{-1}[R[U]]=R[\mathcal{S}_1(h)^{-1}[U]]$ for all $U \in \tau'_1 \cap \delta'_2$. Recall that $\mathcal{S}_1(h)^{-1}[\varphi_+(a)]=\varphi_+(h(a))$ and $\mathcal{S}_1(h)^{-1}[\varphi_-(a)]=\varphi_-(h(a))$. Additionally, we have $\nabla_R\varphi_+(a)=\varphi_+(\nabla a)$ and $\nabla_R\varphi_-(a)=\varphi_-(\nabla a)$ by Theorem \ref{rep thm}. Hence:  
 $$\mathcal{S}_1(h)^{-1}[R[\varphi_+(a)]]=\mathcal{S}_1(h)^{-1}[\nabla_R \varphi_+(a)]=\mathcal{S}_1(h)^{-1}[\varphi_+(\nabla (a))]=\varphi_+(h (\nabla a)).$$
By Definition \ref{De Morgan homomorphism} (2) this equals
$$\varphi_+ (\nabla (h(a)))=\nabla_R \varphi_+(h(a))=\nabla_R(h^{-1})^{-1}[\varphi_+(a)]=\nabla_R \mathcal{S}_1(h)^{-1}[\varphi_+(a)].$$
Analogously, it holds that $\mathcal{S}_1(h)^{-1}[\nabla_R \varphi_-(a)]=\nabla_R \mathcal{S}_1(h)^{-1}[\varphi_-(a)]$.  
\end{proof}

\begin{lemma} The operations
$\mathcal{A}_F=\langle\mathcal{A}_0,\mathcal{A}_1\rangle\colon\mathbf{PStone_{S4D}}\to\mathbf{S4D}$ and $\mathcal{S}_F=\langle\mathcal{S}_0,\mathcal{S}_1\rangle\colon\mathbf{S4D}\to\mathbf{PStone_{S4D}}$ provide fully faithful contravariant functors.  
\end{lemma}
\begin{proof}
    By Lemma \ref{homo to cont} and Lemma \ref{cont to homo} together with Lemma \ref{space to algebra} and Lemma \ref{algebra to space}, it follows that $\langle\mathcal{A}_0,\mathcal{A}_1\rangle\colon\mathbf{PStone_{S4D}}\to\mathbf{S4D}$ and $\langle\mathcal{S}_0,\mathcal{S}_1\rangle\colon\mathbf{S4D}\to\mathbf{PStone_{S4D}}$ are contravariant functors. Hence it suffices to show that if $A$ and $A'$ are S4 De Morgan algebras and $X = \langle X, g, R, \tau_1, \tau_2 \rangle$ and $X' = \langle X', g', R', \tau'_1, \tau'_2 \rangle$ are PS4D-spaces, then: 
    \[\Hom_{\mathbf{S4D}}(A,A')\to\Hom_{\mathbf{PStone_{S4D}}}(\mathcal{S}_0(A'),\mathcal{S}_0(A))\]
     \[\Hom_{\mathbf{PStone_{S4D}}}(X,X')\to\Hom_{\mathbf{S4D}}(\mathcal{A}_0(X'),\mathcal{A}_0(X))\] are fully faithful. For the former, we start by proving the injectivity on morphisms.

     Let $h_1,h_2 \in \Hom_{\mathbf{S4D}}(A,A')$. Assume $h_1 \neq h_2$. Then there is $a \in A$ such that $h_1(a) \neq h_2(a)$. Without loss of generality, let us assume $h_1(a) \nleq h_2(a)$.
     Consider: 
     \[{\uparrow} h_1(a)=\{b \in A' : h_1(a) \leq b\},\hspace{.2cm}{\downarrow} h_2(a)=\{b \in A' :  b \leq h_2(a)\}.\]
     Since $h_1(a) \nleq h_2(a)$, we have ${\uparrow} h_1(a) \cap {\downarrow} h_2(a)= \emptyset$. By Prime Filter Theorem i.e., Theorem \ref{prime filter theorem}, there is a prime filter $x_\rho$ such that ${\uparrow} h_1(a) \subseteq x_\rho$ and ${\downarrow} h_2(a) \cap x_\rho = \emptyset$. Since $h_1(a) \in {\uparrow} h_1(a)$ it follows $h_1(a) \in x_\rho$ and $h_2(a) \notin x_\rho$.
     In other words we have $a \in h_1^{-1}[x_\rho] = \mathcal{S}_1(h_1)(x_\rho)$ and $a \notin h_2^{-1}[x_\rho] = \mathcal{S}_1(h_2)(x_\rho)$. Therefore, $\mathcal{S}_1(h_1) \neq \mathcal{S}_1(h_2)$.

     To show the injectivity of the other functor we assume $f_1,f_2\in \Hom_{\mathbf{PStone}_{\mathbf{S4D}}}(X,X')$ are bicontinuous frame morphisms and there is $x \in X$ such that $f_1(x) \neq f_2(x)$. Since $X'$ is pairwise zero-dimensional and pairwise Hausdorff, by Lemma \ref{l:injectivity} it follows there is $U \in \tau'_1 \cap \delta'_2$, such that $f_1(x) \in U$ and $f_2(x) \notin U$. Hence, it follows 
    $x\in f_1^{-1}[U]=\mathcal{A}_1(f_1)(U)$ and $x \notin f_2^{-1}[U]=\mathcal{A}_1(f_2)(U)$.
    This shows $\mathcal{A}_1(f_1) \neq \mathcal{A}_1(f_2)$.

For surjectivity, note that by Theorem \ref{rep thm}, Theorem \ref{algebraic realization}, Lemma \ref{homo to cont}, and Lemma \ref{cont to homo}, for every $f \in \Hom_{\mathbf{PStone_{S4D}}}(\mathcal{S}_0(A'), \mathcal{S}_0(A))$, there exists $h\in\Hom_{\mathbf{S4D}}(A,A')$ such that $\mathcal{S}_1(h)=f$ and  for every $h \in \Hom_{\mathbf{S4D}}(\mathcal{A}_0(X'), \mathcal{A}_0(X))$, there exists $f\in\Hom_{\mathbf{PStone_{S4D}}}(X,X')$ such that $\mathcal{A}_1(f)=h$ and hence $\mathcal{S}_F$ and $\mathcal{A}_F$ are fully faithful.
\end{proof}

\begin{lemma}
    Let $X = \langle X, g, R, \tau_1, \tau_2 \rangle$ and $X' = \langle X', g', R', \tau'_1, \tau'_2 \rangle$ be PS4D-spaces and let $f\colon X\to X'$ be a bicontinuous frame morphism, then $\mathcal{S}_1(\mathcal{A}_1(f))[\psi(x)]=\psi(f(x))$ for all $x\in X$. 
\end{lemma}
    \begin{proof}
    A simple calculation from the definition of $\mathcal{S}_1$ and $\mathcal{A}_1$ yields: 
    \[\mathcal{S}_1(\mathcal{A}_1(f))[\psi(x)]=(f^{-1})^{-1}[\psi(x)]=\{U\in\tau'_1\cap\delta'_2:f^{-1}[U]\in\psi(x)\}\] Unraveling the definition of preimages and the definition of $\psi$ then gives: 
    \[\{U\in\tau'_1\cap\delta'_2:f^{-1}[U]\in\psi(x)\}=\{U\in\tau'_1\cap\delta'_2:f(x)\in U\}=\psi(f(x))\] which completes the proof. 
\end{proof}

\begin{theorem}\label{main theorem}
    $\mathbf{S4D}$ is dually equivalent to $\mathbf{PStone_{S4D}}$. 
\end{theorem}

\begin{proof}
    We have shown that there exist natural transformations $\varphi_+:\mathbf{id}_{\textbf{S4D}} \to \mathcal{A}_F \circ \mathcal{S}_F$ and $\psi:\textbf{id}_{\mathbf{PStone_{S4D}}} \to \mathcal{S}_F \circ \mathcal{A}_F$, which are natural isomorphisms making:

\centering
 \begin{tikzcd}
A \arrow[dd, "h"'] \arrow[r] \arrow[r] \arrow[rr, bend left, "\phi_+"] & \mathcal{S}_0(A) \arrow[r]                                                & \mathcal{A}_0\circ\mathcal{S}_0(A) \arrow[dd, "\mathcal{A}_1\circ\mathcal{S}_1(h)" description] \\
{} \arrow[r, "1-1", no head, dashed]                                       & {} \arrow[r, "1-1", no head, dashed]                                                 & {}                                                                                          \\
B \arrow[r] \arrow[rr, bend right, "\phi_+"']                             & \mathcal{S}_0(B) \arrow[uu, "\mathcal{S}_1(h)" description] \arrow[r] & \mathcal{A}_0\circ\mathcal{S}_0(B)                    
\end{tikzcd}
\hspace{6pt}
\begin{tikzcd}
X \arrow[dd, "f"'] \arrow[r] \arrow[rr, bend left, "\psi"] & \mathcal{A}_0(X) \arrow[r]                                                & \mathcal{S}_0\circ\mathcal{A}_0(X) \arrow[dd, "\mathcal{S}_1\circ\mathcal{A}_1(f)" description] \\
{} \arrow[r, "1-1", no head, dashed]                            & {} \arrow[r, "1-1", no head, dashed]                                                & {}                                                                                         \\
Y \arrow[r] \arrow[rr, bend right, "\psi"']                   & \mathcal{A}_0(Y) \arrow[uu, "\mathcal{A}_1(f)" description] \arrow[r] & \mathcal{S}_0\circ\mathcal{A}_0(Y)                                            
\end{tikzcd}

commute. This completes the proof.  
\end{proof}
The obtained bitopological duality for S4 De Morgan algebras clearly results in a bitopological duality for general De Morgan algebras under the following restriction. 
\begin{definition}\label{pairwise de morgan stone space}
    If $\langle X;g,R,\tau_1,\tau_2\rangle$ is a pairwise S4 De Morgan Stone space, we call its $R$-free reduct a \emph{pairwise De Morgan Stone space} whenever: 
    \begin{enumerate}
        \item $\tuple{X;\tau_1,\tau_2}$ is a pairwise Stone space; 
        \item $g\colon X\to X$ is a twist continuous involution. 
    \end{enumerate}
\end{definition}
\begin{theorem}
    $\mathcal{S}_0(A)=\langle\mathfrak{P}(A);g_A;\tau_+(\beta_+),\tau_-(\beta_-)\rangle$ is a pairwise De Morgan Stone space whenever $A$ is a De Morgan algebra. Moreover, $\mathcal{A}_0(X)=\langle\tau_1\cap\delta_2;\cap,\cup,^*,\emptyset,X\rangle$ is a De Morgan algebra whenever $X$ is a pairwise De Morgan Stone space. 
\end{theorem}
\begin{proof}
    Immediate by Lemma \ref{algebra to space} and Lemma \ref{space to algebra}. 
\end{proof}
\begin{definition}
    If $X = \langle X, g, R, \tau_1, \tau_2 \rangle$ and $X' = \langle X', g', R', \tau'_1, \tau'_2 \rangle$ are pairwise De Morgan spaces, a function $f\colon X\to X'$ is a \emph{bicontinuous frame morphism} if $f$ is a bicontinuous function satisfying $f(g(x))=g'(f(x))$. 
  
\end{definition}
\begin{definition}
    By $\mathbf{DM}$ we denote the category of De Morgan algebras and De Morgan algebra homomorphisms. By $\mathbf{PStone_{DM}}$ we denote the category of pairwise De Morgan Stone spaces and bicontinuous frame morphisms. 
\end{definition}
\begin{theorem}
    $\mathbf{DM}$ is dually equivalent to $\mathbf{PStone_{DM}}$. 
\end{theorem}
\begin{proof}
    Immediate by Theorem \ref{main theorem}.  
\end{proof}

\subsection{$\mathbf{Priest_{DM}}$ vs. $\mathbf{PStone_{DM}}$ and $\mathbf{Spec_{DM}}$} It is worth mentioning a subtle distinction between Priestley duality for De Morgan algebras and both pairwise Stone duality and spectral duality for De Morgan algebras. Recall from \cite{cornish1} that a \emph{De Morgan Priestley space} is a topological space $\langle X;\leq,g,\tau\rangle$ such that $\langle X;\leq,\tau\rangle$ is a compact partially ordered topological space such that if $x\not\leq y$, there exists a clopen upset $U$ of $X$ such that $x\in U$ and $y\not\in U$. Moreover, $g\colon X\to X$ is a continuous order-inverting involution. It follows that the prime filter spectrum $\mathcal{S}_0(A)=\langle\mathfrak{P}(A);\subseteq,g_A,\tau(\sigma)\rangle$ with sub-basis 
$$\sigma = \{\phi(a):a\in A\}\cup \{\mathfrak{P}(A)\setminus\phi(a) : a \in A\}$$
is a De Morgan Priestley space, the algebra of clopen upsets $$\mathcal{A}_0(X)=\langle\delta\cap\tau\cap\mu;\cap,\cup,^*,\emptyset,X\rangle$$ is a De Morgan algebra, where $\mu$ is the collection of upsets of $X$, and that every De Morgan algebra $A$ is isomorphic to $\mathcal{A}_0(\mathcal{S}_0(A))$ under $\phi(a)=\{x\in\mathfrak{P}(A):a\in x\}$.

In \cite{mcdonald}, a \emph{De Morgan spectral space} is defined as a topological space $\langle X;g,\tau\rangle$ such that $\langle X;\tau\rangle$ is a spectral space (i.e., $X$ is compact, coherent, sober, and $T_0$), $g\colon X\to X$ is an order-inverting involution, and $U^*$ is compact open whenever $U$ is compact open. Then the prime filter spectrum $\mathcal{S}_0(A)=\langle\mathfrak{P}(A);g_A,\tau(\beta)\rangle$ with basis $\beta=\{\phi(a) : a \in A\}$ is a De Morgan spectral space, the algebra of compact open sets $\mathcal{A}_0=\langle\sigma\cap\tau;\cap,\cup,^*,\emptyset,X\rangle$ is a De Morgan algebra, and every De Morgan algebra $A$ is isomorphic to $\mathcal{A}_0(\mathcal{S}_0(A))$ under $\phi$. Hence, an important distinction between these dualities, aside from the distinction in the Priestley, pairwise Stone, and spectral topologies underlying the respective dual spaces, is the requirement in Priestley duality that $g$ be continuous, which implies that $U^*$ is a clopen upset whenever $U$ is a clopen upset. The analogs for pairwise Stone spaces under bicontinuous functions as well as spectral spaces under spectral maps are however insufficient to guarantee that $U^*$ belongs to the corresponding algebra of sets. This is precisely described by the following.           

\begin{proposition}
    Let $\langle X;\leq,\tau\rangle$ be a Priestley space, let $\langle X;\tau_1,\tau_2\rangle$ be a pairwise Stone space, let $\langle X,\tau\rangle$ be a spectral space, and let $U^*=\{x\in X:g(x)\not\in U\}$. Then: 
    \begin{enumerate}
        \item if $U$ is a clopen upset in $\langle X;\leq,\tau\rangle$ and $g\colon X\to X$ is a continuous decreasing map, then $U^*$ is a clopen upset; 
        \item if $U$ is $(\tau_1,\delta_2)$-biclopen in $\langle X,\tau_1,\tau_2\rangle$ and $g\colon X\to X$ is a bicontinuous map, then $U^*$ is not necessarily $(\tau_1,\delta_2)$-biclopen; 
        \item if $U$ is compact open in $\langle X;g,\tau\rangle$ and $g\colon X\to X$ is a spectral map, then $U^*$ is not necessarily compact open. 
    \end{enumerate}
\end{proposition}
\begin{proof}
    For part 1, suppose $U$ is a clopen upset in $\langle X;\leq,\tau\rangle$, then we have 
    \[U^*=\{x\in X:g(x)\not\in U\}=X\setminus g^{-1}[U]=g^{-1}[X\setminus U].\] Clearly $g^{-1}[X\setminus U]$ is clopen in $\langle X;\leq,g,\tau\rangle$ since $g$ is continuous. Now assume that $x\in g^{-1}[X\setminus U]$ so that $g(x)\in X\setminus U$ and hence $g(x)\not\in U$. Take any $y\in X$ such that $x\leq y$. Since $g$ is decreasing, we have $g(y)\leq g(x)$ and hence $g(y)\not\in U$ so $g(y)\in X\setminus U$ and therefore $y\in g^{-1}[X\setminus U]$. Thus $g^{-1}[X\setminus U]$ is also an upset and hence $U^*$ a clopen upset. For condition 2 observe that if $U$ is $(\tau_1,\delta_2)$-biclopen in $\langle X;g,\tau_1,\tau_2\rangle$, then $U\in\tau_1\cap\delta_2$ so $U\in\tau_1$ and $U\in\delta_2$ and hence $X\setminus U\in\delta_1$ and $X\setminus U\in\tau_2$. Therefore $g^{-1}[X\setminus U]\in\delta_1$ and $g^{-1}[X\setminus U]\in\tau_2$ so $g^{-1}[X\setminus U]=U^*\in\delta_1\cap\tau_2\not=\tau_1\cap\delta_2$. For condition 3, note that $g^{-1}[X\setminus U]$ needn't be compact open when $U$ is compact open in $\langle X;\tau\rangle$.    
\end{proof}
This justifies our requirement in condition 2 of Definition \ref{pairwise S4 De Morgan space} and condition 2 of Definition \ref{pairwise de morgan stone space} that $g$ be twist continuous as well as the requirement in \cite{mcdonald} that $U^*$ be compact open whenever $U$ is compact open. As we have seen, the analogue of this requirement in the setting of Priestley duality is implied by assuming $g$ is a continuous decreasing function.

\section{Applications} \label{sec:applications}
In this section, we investigate various applications of our established duality in the setting of S4 De Morgan algebras as well as the model theory of an S4-type modal extension of the calculus FDE.

\subsection{Filters and ideals} Following \cite{bezhanishvili}, we provide bitopological characterizations of filters and ideals in the setting of De Morgan algebras under our established duality. In doing so, we lift various well-known facts about Priestley duality for bounded distributive lattices to our bitopological duality for De Morgan algebras.

\begin{proposition}\label{proposition 5.1}
    Let $A$ be a De Morgan algebra, let $\mathfrak{F}(A)$ be the filters of $A$, let $\mathfrak{I}(A)$ be the ideals of $A$, and let $X$ be the De Morgan Priestley space dual to $A$. Then:
    \begin{enumerate}
        \item $\langle\mathfrak{F}(A);\supseteq\rangle$ is isomorphic to  $\langle\delta\cap\mu;\subseteq\rangle$; 
        \item $\langle\mathfrak{I}(A);\subseteq\rangle$ is isomorphic to $\langle\tau\cap\mu;\subseteq\rangle$. 
    \end{enumerate}   
\end{proposition}
\begin{proof}
    For part 1, associate with each filter $x\in\mathfrak{F}(A)$, the closed upset $\bigcap\{\phi(a):a\in x\}$ and with each closed upset $U\subseteq X$, associate the filter $\{a\in A:U\subseteq\phi(a)\}$. For part 2, associate with each ideal $y\in\mathfrak{I}(A)$, the open upset $\bigcup\{\phi(a):a\in y\}$ and with each open upset $V\subseteq X$, associate the ideal $\{a\in A:\phi(a)\subseteq V\}$.   
\end{proof}

\begin{proposition}\label{prop 5.5}
    Let $\langle X;\leq,g,\tau\rangle$ be a De Morgan Priestley space and let $\langle X;g,\tau_1,\tau_2\rangle$ be its corresponding pairwise De Morgan Stone space. Then: 
    \begin{enumerate}
        \item $U$ is a closed upset of $\langle X;\leq,g,\tau\rangle$ if and only if $U$ is a $\tau_2$-closed set of $\langle X;g,\tau_1,\tau_2\rangle$; 
        \item $U$ is a closed downset of $\langle X;\leq,g,\tau\rangle$ if and only if $U$ is a $\tau_1$-closed set of $\langle X;g,\tau_1,\tau_2\rangle$.  
    \end{enumerate}
\end{proposition}
\begin{proof}
 The result follows from \cite[Proposition 3.4]{bezhanishvili} however we outline the proof. For part 1, note that each closed upset is the intersection of clopen upsets in $\langle X;\leq,g,\tau\rangle$, clopen upsets are in $\beta_1$, and hence it follows that closed upsets are intersections of elements from $\beta_1$. As $\beta_1=\{X\setminus U:U\in\beta_2\}$, it follows that the intersection of elements of $\beta_1$ are intersections of complements of elements of $\beta_2$. Since $\beta_2$ is a basis, they are complements of unions of elements of $\beta_2$ and thus unions of such elements are in $\tau_2$. Hence, since closed upsets are complements of elements in $\tau_2$, they are in $\delta_2$. The proof of part 2 follows analogously to the proof of part 1.      
\end{proof}

\begin{corollary}
      Let $A$ be a De Morgan algebra, let $\mathfrak{F}(A)$ be the collection of all filters of $A$, let $\mathfrak{I}(A)$ be the collection of all ideals of $A$, and let $X$ be the pairwise De Morgan Stone space dual to $A$. Then:
    \begin{enumerate}
        \item $\langle\mathfrak{F}(A);\supseteq\rangle$ is isomorphic to $\langle\delta_2;\subseteq\rangle$;   
        \item $\langle\mathfrak{I}(A);\subseteq\rangle$ is isomorphic to $\langle\tau_1;\subseteq\rangle$. 
    \end{enumerate}
\end{corollary}
\begin{proof}  
The result follows immediately from Propositions \ref{proposition 5.1} and \ref{prop 5.5}. 
\end{proof}

\subsection{Bitopological completeness of $G_{S4-FDE}$}

In this subsection, we briefly introduce the logical counterpart of S4 De Morgan algebras by providing the corresponding Gentzen calculus and sketching the completeness theorem. Before proceeding, let us formalize the notion of validity in the extension of De Morgan algebras.

\begin{definition}
    Let $A$ be an algebra in a language $\mathcal{L}$ with a De Morgan algebra reduct, and let $\Gamma \cup \{\psi\}$ be a finite set of formulas in $\mathcal{L}$. 

    \begin{enumerate}
        \item We say that a sequent $\Gamma \vdash \psi$ is \textit{valid in $A$}, denoted $\Gamma \vDash_{A} \psi$, if for every evaluation $e: \mathit{Form}_{\mathcal{L}} \to A$, it holds that:
        \[ \bigwedge_{\varphi \in \Gamma} e(\varphi) \leq e(\psi) \]
        (By convention, if $\Gamma = \emptyset$, then $\bigwedge \emptyset = \top^{A}$).
        
        \item For a class of algebras $\mathbb{K}$, we say the sequent $\Gamma \vdash \psi$ is \textit{valid in $\mathbb{K}$}, denoted $\Gamma \vDash_{\mathbb{K}} \psi$, if $\Gamma \vDash_{A} \psi$ for every algebra $A \in \mathbb{K}$.
        
        \item A Gentzen calculus $G$ is said to be \textit{sound and complete} with respect to the class $\mathbb{K}$ if, for every finite set of formulas $\Gamma$ and formula $\psi$:
        \[ \Gamma \vdash_G \psi \iff \Gamma \vDash_{\mathbb{K}} \psi. \]
    \end{enumerate}
\end{definition}

S4 De Morgan algebras correspond to the S4-modal extension of the bounded version of Belnap-Dunn logic. The original Belnap-Dunn logic, often referred to as First-Degree Entailment (FDE), was introduced in \cite{belnap,Dunn:SemanticsBD} as the logic of the four-element De Morgan lattice $\mathbf{DM_4}$ (presented in Figure \ref{fig:DM4}) and has since been intensively studied (for a recent overview of the field, see \cite{Omori-Walsing:40FDE}). Apart from its four-valued semantics, this logic gained fame as a system that does not contain any theorems.

Here, we recall the Gentzen calculus for Belnap-Dunn logic introduced in \cite{Font:BD} (originally disguised as a Hilbert calculus in \cite{Anderson-Belnap:Entailment}), supplemented with additional rules for the bounds $\top$ and $\bot$.
\begin{table}[htbp]
\centering
\label{f:GB}
\begin{align*}
\dfrac{}{\Gamma, \varphi \vdash \varphi} \quad (Id) 
&\qquad\qquad 
\dfrac{\Gamma \vdash \varphi}{\Gamma, \psi \vdash \varphi} \quad (W) 
\\[1em]
\dfrac{\Gamma, \varphi, \psi \vdash \xi}{\Gamma, \varphi \land \psi \vdash \xi} \quad (\land \vdash) 
&\qquad\qquad
\dfrac{\Gamma \vdash \varphi \quad \Gamma \vdash \psi}{\Gamma \vdash \varphi \land \psi} \quad (\vdash \land)
\\[1em]
\dfrac{\Gamma, \varphi \vdash \xi \quad \Gamma, \psi \vdash \xi}{\Gamma, \varphi \lor \psi \vdash \xi} \quad (\lor \vdash) 
&\qquad\qquad
\dfrac{\Gamma \vdash \varphi}{\Gamma \vdash \varphi \lor \psi} \, , \, \dfrac{\Gamma \vdash \psi}{\Gamma \vdash \varphi \lor \psi} \quad (\vdash \lor)
\\[1em]
\dfrac{\varphi \vdash \psi}{\neg \psi \vdash \neg \varphi} \quad (\neg) 
&\qquad\qquad
\dfrac{\Gamma \vdash \varphi \quad \Gamma, \varphi \vdash \psi}{\Gamma \vdash \psi} \quad (\text{Cut})
\\[1em]
\dfrac{\Gamma, \varphi \vdash \psi}{\Gamma, \neg\neg\varphi \vdash \psi} \quad (\neg\neg \vdash) 
&\qquad\qquad
\dfrac{\Gamma \vdash \varphi}{\Gamma \vdash \neg\neg\varphi} \quad (\vdash \neg\neg)
\\[1em]
\dfrac{}{\Gamma, \bot \vdash \varphi} \quad (\bot \vdash) 
&\qquad\qquad
\dfrac{}{\Gamma \vdash \top} \quad (\vdash \top)
\end{align*}
\caption{The rules of the sequent calculus $G_{FDE}$}
\end{table}

Let us note that while the Gentzen calculus $G_{FDE}$ is the most commonly mentioned formulation, other versions of Gentzen calculi, more suitable for proof-theoretic purposes, exist as well. For instance, \cite{Pynko:Genzen} provides a cut-free Gentzen calculus for Belnap-Dunn logic. Furthermore, Hilbert-style calculi for Belnap-Dunn logic are well established \cite{Font:BD,Shramko:HilbertFDE}; however, these systems are trickier to adapt to modal extensions, as they cannot explicitly capture many of the necessary structural inference rules.

We now introduce the additional rules for the modal operator $\Diamond$, which, combined with $G_{FDE}$, yield the calculus $G_{S4-FDE}$ for S4 De Morgan algebras.

\begin{table}[htbp]
\centering
\label{t:GS4-FDE}
\begin{gather*}
\begin{aligned}
\dfrac{}{\Gamma,\Diamond \bot \vdash \varphi} \quad (\Diamond \bot)
&\qquad\qquad
\dfrac{\varphi \vdash \psi}{\Diamond \varphi \vdash \Diamond \psi} \quad (\Diamond\text{ Mon}) 
\\[1em]
\dfrac{\Gamma \vdash \varphi}{\Gamma \vdash \Diamond\varphi} \quad (\Diamond \text{ T})
&\qquad\qquad
\dfrac{\Gamma, \Diamond\varphi \vdash \psi}{\Gamma, \Diamond\Diamond\varphi \vdash \psi} \quad (\Diamond \text{ 4})
\end{aligned}
\\[1.5em]
\dfrac{\Gamma, \Diamond\varphi \lor \Diamond\psi \vdash \xi}{\Gamma, \Diamond(\varphi \lor \psi) \vdash \xi} \quad (\Diamond \lor)
\end{gather*}
\caption{The additional rules for $\Diamond$ of the sequent calculus $G_{S4-FDE}$}
\end{table}
\begin{theorem}\label{algebraic completeness}
    The calculus $G_{S4-FDE}$ is sound and complete with respect to the class of S4 De Morgan algebras.
\end{theorem}

\begin{proof}
    It is straightforward to check that all S4 De Morgan algebras satisfy the rules and axioms of $G_{S4-FDE}$. For the other implication, we proceed with the standard construction of the Lindenbaum–Tarski algebra. Let us define a relation $\theta$ on the set of formulas $\mathit{Form}_{\mathcal{L}}$ by $\varphi \mathrel{\theta} \psi$ if and only if $\varphi \vdash \psi$ and $\psi \vdash \varphi$. Clearly, by the rules \rul{Id} and \rul{Cut}, this is an equivalence relation. Let $A = \mathit{Form}_{\mathcal{L}} / \theta$ be the quotient set of formulas factored by this equivalence relation. We define an algebra $\mathbf{A}$ with universe $A$ and the following operations:
\begin{multicols}{2}
\begin{enumerate}
    \item $[\varphi] \lor_{\mathbf{A}} [\psi] = [\varphi \lor \psi]$; 
    \item $[\varphi] \land_{\mathbf{A}} [\psi] = [\varphi \land \psi]$; 
    \item $-[\varphi] = [\neg \varphi]$; 
    \item $0 = [\bot]$; 
    \item $1 = [\top]$; 
    \item $\nabla[\varphi] = [\Diamond \varphi]$.
\end{enumerate}
\end{multicols}
    Note that the operations $\lor_{\mathbf{A}}, \land_{\mathbf{A}}, -, \nabla$ are well-defined. For the propositional connectives, this can be easily verified using the rules \rul{\vdash \land}, \rul{\land \vdash}, \rul{\vdash \lor}, \rul{\lor \vdash}, and \rul{\neg}. For the modal operator $\nabla$, well-definedness follows directly from the monotonicity rule \rul{\Diamond\text{ Mon}}.

    Next, we consider the structure of the algebra $\mathbf{A}$. It can be easily checked that $\mathbf{A}$ is a De Morgan algebra (alternatively, this follows from the fact that the calculus $G_{S4-FDE}$ expands the base calculus $G_{FDE}$, which is known to be complete with respect to the class of De Morgan algebras; see \cite{Font:BD}). 

    First, we verify that $\mathbf{A}$ is a \emph{modal} De Morgan algebra. To verify this, we need to check that $\nabla (a \lor b) = \nabla a \lor \nabla b$ and $\nabla 0 = 0$. The inequality $\nabla (a \lor b) \geq \nabla a \lor \nabla b$ follows from \rul{Id}, \rul{\vdash \lor}, \rul{\lor \vdash}, and \rul{\Diamond\text{ Mon}}, while the reverse inequality $\nabla (a \lor b) \leq \nabla a \lor \nabla b$ follows from \rul{\Diamond \lor}. The identity $\nabla 0 = 0$ follows from \rul{\Diamond \bot} (which gives $\nabla 0 \leq 0$) and \rul{\bot \vdash} (which guarantees $0 \leq \nabla 0$). Hence, $\mathbf{A}$ is a modal De Morgan algebra. To check that $\mathbf{A}$ is an S4 De Morgan algebra, it remains to verify the conditions $a \leq \nabla a$ and $\nabla \nabla a \leq \nabla a$ (which implies $\nabla a = \nabla \nabla a$). However, these inequalities easily follow from the rules \rul{\Diamond T} and \rul{\Diamond 4}, respectively.
\end{proof}

As a brief application of the obtained duality results, we provide bitopological soundness and completeness results for the calculus S4-FDE. 
\begin{definition}
    Let $X$ be a PS4D-space, let $\text{Var}=\{x_i:i<\omega\}$ be a countable set of variables, and let $v_X\colon\text{Var}\to\mathcal{A}_0(X)$ be a map which sends every $x_i\in\text{Var}$ to some $(\tau_1,\delta_2)$-biclopen set $U\in\mathcal{A}_0(X)$. Then extend $v_X$ to a map $\widehat{v_X}\colon \mathit{Form}_{\mathcal{L}}\to\mathcal{A}_0(X)$ according to the following inductive definition: 
    \begin{enumerate}
        \item $\widehat{v_X}(\bot)=\emptyset$ and $\widehat{v_X}(\top)=X$; 
        \item $\widehat{v_X}(\phi\wedge\psi)=\widehat{v_X}(\phi)\cap\widehat{v_X}(\psi)$; 
        \item $\widehat{v_X}(\phi\vee\psi)=\widehat{v_X}(\phi)\cup\widehat{v_X}(\psi)$; 
        \item $\widehat{v_X}(\neg\phi)=\widehat{v_X}(\phi)^*$; 
        \item $\widehat{v_X}(\Diamond \phi)=\nabla_R\widehat{v_X}(\phi).$
    \end{enumerate}
\end{definition}
\begin{definition}
    Given a finite set of formulas $\Gamma \cup\{\psi\}$ in $\mathcal{L}$, define $\models\subseteq\mathit{Form}_{\mathcal{L}} \times \mathit{Form}_{\mathcal{L}} $: 
    \[\Gamma\models\psi\Longleftrightarrow \bigcap_{\varphi \in \Gamma}\widehat{v_X}(\phi)\subseteq\widehat{v_X}(\psi),\]
    where, by convention, if $\Gamma = \emptyset$, we define $\bigcap_{\varphi \in \Gamma} \widehat{v_X}(\varphi)=X$.
\end{definition}
Let us note that since $\Gamma$ is a \emph{finite} set of formulas, the set $\bigcap_{\varphi \in \Gamma}\widehat{v_X}(\phi)$ is biclopen.

\begin{theorem}
    $G_{S4-FDE}$ is sound and complete with respect to the class of all PS4D-spaces. 
\end{theorem}
\begin{proof}
    By Theorem \ref{algebraic completeness}, the calculus $G_{S4-FDE}$ is sound and complete with respect to the class of S4 De Morgan algebras. By Theorem \ref{main theorem}, every S4 De Morgan algebra $\alg A$ is isomorphic to $\mathcal{A}_0(X)$ where $X=\mathcal{A}_0(X)$ is the PS4D-space dual to $\alg A$.
\end{proof}

\section{Conclusions and future lines of research}
In this work, we have investigated the duality theory of S4 De Morgan algebras from the bitopological perspective of pairwise Stone spaces. We have also explored various applications of the obtained duality. 

Related future lines of research may involve extending our obtained duality results to the setting of De Morgan groupoids, as described in \cite{mcdonald}. The duality theory of such algebraic structures would involve translating the spectral duality for De Morgan groupoids developed in \cite{mcdonald} under the duality established here. Additionally, we believe there are several aspects which deserve further investigation. It is well known that De Morgan Algebras can be equivalently algebraically represented as twisted products of distributive lattices. This specific form of representation very naturally corresponds to the bitopological representation of S4 De Morgan Algebras investigated in this paper. However, the precise connection here remains to be established.

\subsubsection*{Acknowledgments}
We thank the Department of Theoretical Computer Science at the Institute of Computer Science of the Czech Academy of Sciences for their support throught the preparation of this work. 

\subsubsection*{Funding statement}
This work has been funded by a grant from the Programme Johannes Amos Comenius
under the Ministry of Education, Youth and Sports of the Czech Republic, grant no.
CZ.02.01.01/00/23$_{-}$025/0008711.

\end{document}